\documentclass[a4paper,11pt]{amsart}
\usepackage[english]{babel}
\usepackage[latin1]{inputenc}
\usepackage{amssymb}
\usepackage{amscd}
\usepackage{mathtools}
\usepackage[paper=a4paper,left=3cm,right=3cm,top=3cm,bottom=2.3cm]{geometry}
\usepackage{enumerate}
\usepackage{url}
\usepackage{hyperref}

\newtheorem{thm}{Theorem}[section]
\newtheorem*{thm-non}{Theorem}
\newtheorem{lem}[thm]{Lemma}
\newtheorem{prop}[thm]{Proposition}
\newtheorem{cor}[thm]{Corollary}
\theoremstyle{definition}
\newtheorem{defi}[thm]{Definition}
\theoremstyle{remark}

\DeclareMathOperator{\Schemes}{Schemes}
\DeclareMathOperator{\Sets}{Sets}
\DeclareMathOperator{\Spec}{Spec}
\DeclareMathOperator{\Hilb}{Hilb}
\DeclareMathOperator{\Quot}{Quot}
\DeclareMathOperator{\Sym}{Sym}
\DeclareMathOperator{\HH}{H}
\DeclareMathOperator{\Hom}{Hom}
\DeclareMathOperator{\Ext}{Ext}
\DeclareMathOperator{\HHom}{\mathcal{H}om}
\DeclareMathOperator{\tr}{tr}
\DeclareMathOperator{\rk}{rk}
\DeclareMathOperator{\supp}{supp}
\DeclareMathOperator{\sing}{sing}
\DeclareMathOperator{\lf}{lf}
\DeclareMathOperator{\Pic}{Pic}
\DeclareMathOperator{\Br}{Br}
\DeclareMathOperator{\NS}{NS}

\newcommand{\C}{\mathbb{C}}
\newcommand{\Q}{\mathbb{Q}}
\newcommand{\Z}{\mathbb{Z}}
\newcommand{\N}{\mathbb{N}}
\newcommand{\PP}{\mathbb{P}}
\newcommand{\A}{\mathcal{A}}
\newcommand{\Ahat}{\widehat{\A}}
\newcommand{\E}{\mathcal{E}}
\newcommand{\Ehat}{\widehat{E}}
\newcommand{\I}{\mathcal{I}}
\newcommand{\J}{\mathcal{J}}

\newcommand{\M}{\mathcal{M}}
\newcommand{\OO}{\mathcal{O}}
\newcommand{\Ohat}{\widehat{\OO}}
\newcommand{\m}{\mathfrak{m}}
\newcommand{\n}{\mathfrak{n}}
\newcommand{\lmod}{\text{-mod}}

\begin{document}

\title{Torsion-free rank one sheaves over del Pezzo orders}
\author{Norbert Hoffmann}
\address{Department of Mathematics and Computer Studies\\Mary Immaculate College\\South Circular
  Road\\Limerick\\Ireland}
\email{norbert.hoffmann@mic.ul.ie}

\author{Fabian Reede}
\address{Department of Mathematics and Computer Studies\\Mary Immaculate College\\South Circular
  Road\\Limerick\\Ireland}
\email{fabianreede@gmx.net}

\thanks{The second author was supported by a research fellowship of the Deutsche Forschungsgemeinschaft (DFG)}
\maketitle

\begin{abstract}
  Let $\A$ be a del Pezzo order on the projective plane over the field of complex numbers.
  We prove that every torsion-free $\A$-module of rank one can be deformed into
  a locally free $\A$-module of rank one.
\end{abstract}

\section*{Introduction}
An \emph{order} on an algebraic variety $X$ is a torsion-free coherent sheaf of $\OO_X$-algebras whose generic
stalk is a central division algebra over the function field of $X$. A surface together with an order on it
can be thought of as a noncommutative surface. In this article we are interested in terminal del Pezzo orders
on the projective plane $\PP^2$ over the field of complex numbers $\C$. These orders are noncommutative analogues
of classical del Pezzo surfaces and have been completely classified by D.\ Chan and C.\ Ingalls
in the course of their proof of the minimal model program for orders over surfaces, see \cite{chan2}.

Let $\A$ be a terminal del Pezzo order on $\PP^2$. Left $\A$-modules which are locally free and
generically of rank one can be thought of as line bundles on this noncommutative surface.
There is a quasi-projective coarse moduli scheme for these line bundles \cite{hst}, a noncommutative analogue
of the classical Picard scheme. To compactify this moduli scheme, that is to get a projective moduli scheme,
one has to allow torsion-free left $\A$-modules which are generically of rank one.

We study the boundary of this compactification by studying the deformation theory of torsion-free $\A$-modules.
The main result of this article is the following
\begin{thm-non}
  Let $\A \neq \OO_{\PP^2}$ be a terminal del Pezzo order on $\PP^2$ over $\C$.
  Then every torsion-free $\A$-module $E$ of rank one can be deformed to a locally free $\A$-module $E'$.	
\end{thm-non}
As a corollary, we obtain that every irreducible component of the compactification of the noncommutative Picard scheme
contains a point defined by an $\A$-line bundle.

The structure of this paper is as follows.
We review the definition and some basic facts about terminal del Pezzo orders in section \ref{sec1}.
In section \ref{sec2} we study in detail the local deformation theory of $\A$-modules in this setting.
We look at the homological algebra of torsion-free $\A$-modules and study the compactification of
the noncommutative Picard scheme and some of its properties in section \ref{sec3}.
In the final section \ref{sec4} we study the global deformation theory and prove the main result.

%\tableofcontents

\section{Noncommutative del Pezzo surfaces}\label{sec1}
Let $X$ be a smooth projective surface over $\C$. 
\begin{defi}
  An \emph{order} $\A$ on $X$ is sheaf of associative $\OO_X$-algebras such that
  \begin{itemize}
   \item $\A$ is coherent and torsion-free as an $\OO_X$-module, and
   \item the stalk $\A_{\eta}$ at the generic point $\eta\in X$ is a central division ring
    over the function field $\C(X)=\OO_{X,\eta}$ of $X$.
  \end{itemize} 
\end{defi}

We can now look at all orders in $\A_{\eta}$ and order them by inclusion.
A maximal element will be called a maximal order. These are the algebras we are interested in.
Maximal orders have some nice properties, for example they are locally free $\OO_X$-modules.

Furthermore, it is well known that there is a largest open subset $U\subset X$ on which
$\A$ is even an Azumaya algebra, see for example \cite[Proposition 6.2]{tannen}.
The complement $D:=X \setminus U$ is called the ramification locus of $\A$.
It is the union of finitely many curves $C \subset X$,
and contains valuable informations about the order $\A$.

The ramification of a maximal order $\A$ can be seen in the Artin-Mumford sequence:
\begin{thm}[\normalfont{\cite[Lemma 4.1]{tannen}}]
  Let $X$ be a smooth projective surface over $\C$. Then there is a canonical exact sequence
  \begin{equation*} \begin{CD}
    0 @>>> \Br(X) @>>> \Br(\C(X))
      @>>> \bigoplus\limits_{\genfrac{}{}{0pt}{}{C \subset X}{\text{irreducible curve}}}
        \HH^1(\C(C),\Q/\Z).
  \end{CD}\end{equation*}
\end{thm}
Here the Galois cohomology group $\HH^1(\C(C),\Q/\Z)$
classifies isomorphism classes of cyclic extensions of $\C(C)$.
The ramification curves are exactly the curves where the Brauer class of $\A_\eta$
has nontrivial image in $\HH^1(\C(C),\Q/\Z)$.
Thus every ramification curve $C$ comes with a finite cyclic field extension $L/\C(C)$.
The degree $e_C:=[L:\C(C)]$ is called the ramification index of $\A$ at $C$.

We are interested in a special class of maximal orders on $X$, the so called terminal orders.
To give a definition of terminal orders, let $e$, $e'$ and $f$ be positive integers such that $e'$ divides $e$.
We look at the complete local ring $R=\C[[u,v]]$ and define 
\begin{equation*}
  S:= R \langle x,y \rangle \text{ with the relations } x^{e'}=u, y^{e'}=v \text{ and } yx=\zeta xy
\end{equation*}
where $\zeta$ is a primitive $e'$-th root of unity. Then $S$ is of finite rank over $R$, the center of $S$ is $R$,
and the tensor product $S \otimes_R K$ with the field of fractions $K := \Quot( R)$ is a division ring.
Define the following $R$-subalgebra:
\begin{equation}\label{algdef}
  B := \begin{pmatrix*}
    S & \cdots & \cdots & S  \\
    xS & S & \ddots & \vdots\\
    \vdots & \ddots & \ddots & \vdots\\
    xS & \cdots & xS & S
  \end{pmatrix*} \subset M_{e/e'}(S)
\end{equation}
Then we define the $R$-algebra $A$ as a full matrix algebra over $B$:
\begin{equation} \label{eq:def_A}
  A := M_f( B).
\end{equation}
Note that the algebra $A = A_{e, e', f}$ depends on the integers $e$, $e'$ and $f$.
The following theorem describes some of its properties:
\begin{thm}[\normalfont{\cite[Proposition 2.8]{chan2}}]
  Let $A = A_{e, e', f}$ be the $R$-algebra defined by \eqref{eq:def_A}.
  \begin{itemize}
   \item[i)] $A$ has global dimension two.
   \item[ii)] If $e=e'=1$, then $A$ is unramified.
   \item[iii)] If $e>e'=1$, then $A$ is ramified on $u=0$, with ramification index $e$.
   \item[iv)] If $e'>1$, then $A$ is ramified on $uv=0$, with ramification index $e$ on $u=0$,
    and with ramification index $e'$ on $v=0$.
  \end{itemize}
\end{thm}
\begin{defi}[\normalfont{\cite[Corollary 4.3]{chan2}}]\label{locstr}
  A maximal order $\A$ on a smooth projective surface $X$ over $\C$ is called \emph{terminal}
  if and only if for every closed point $p \in X$ there is
  \begin{itemize}
   \item an isomorphism of complete local rings $\Ohat_{X,p} \cong \C[[u,v]]$, and
   \item a $\C[[u,v]]$-algebra isomorphism $\A_p \otimes \Ohat_{X,p} \cong A_{e, e', f}$
    for some integers $e$, $e'$ and $f$.
  \end{itemize}
\end{defi}
\begin{defi}[\normalfont{\cite[Lemma 8]{chan5}}]
  Assume $\A$ is a terminal order on a smooth projective surface $X$ over $\C$,
  with ramification curves $\left\{C_i\right\}$ and ramification indices $\left\{e_i\right\}$.
  Then we define the canonical divisor class $K_\A$ of $\A$ by: 
  \begin{equation*}
    K_\A=K_X+\sum(1-\frac{1}{e_i})C_i.
  \end{equation*}
\end{defi}
\begin{lem} \label{lem:K_A}
  If $\A$ is a terminal order on a smooth projective surface $X$ over $\C$, then
  \begin{equation*}
    K_\A=K_X - \dfrac{2 c_1( \A)}{\rk( \A)}.
  \end{equation*}
\end{lem}
\begin{proof}
  Theorem 1.83 in \cite{reede} states that $\displaystyle c_1( \A) = - \frac{\rk( \A)}{2} \sum(1-\frac{1}{e_i})C_i$.
\end{proof}

\begin{defi}[\normalfont{\cite[Definition 7, Lemma 8]{chan5}}]
  A terminal order $\A$ on a smooth projective surface $X$ over $\C$
  is called a \emph{del Pezzo order} if $-K_\A$ is ample.
\end{defi}
If $\A$ is a terminal del Pezzo order on $\PP^2$, then its ramification is rather limited:
\begin{prop}[\normalfont{\cite[Proposition 3.21]{chan2}}]\label{prop:delpez}
  Assume $\A$ is a terminal del Pezzo order on $\PP^2$ with ramification locus $D=\bigcup C_i$
  and ramification indices $\left\{e_i\right\}$. Then all ramification indices $e_i$ are equal,
  and we have $3\leq deg(D)\leq 5$.
\end{prop}
Furthermore there are more constraints for the common ramification index $e \in \N$
depending on the degree of $D$, see for example \cite[Proposition 3.21]{chan2}.

\section{Punctual deformations of rank one modules}\label{sec2}
In this section we study the local situation. That is we replace the surface $X$ over $\C$
by the complete local ring $R = \C[[u,v]]$, and the terminal order $\A$ on $X$
by the $R$-algebra
\begin{equation*}
  A = A_{e, e', f}
\end{equation*}
defined in \eqref{eq:def_A}. The role of dualizing sheaf will be played by the $A$-bimodule
\begin{equation*}
  A^* := \Hom( A, R).
\end{equation*}
Left ideals $I \subset A$ of $R$-colength $l < \infty$ are parameterized by the punctual Hilbert scheme
\begin{equation*}
  \Hilb_A( l),
\end{equation*}
which is a closed subscheme of the punctual Quot-scheme $\Quot_R(A,l)$ and hence projective over $\C$.
We say that $I$ can be \emph{deformed} to another left ideal $I' \subset A$ if
$I'$ has the same colength $l < \infty$, and lies in the same connected component of $\Hilb_A( l)$.

Equivalently, $I \subset A$ can be deformed to $I' \subset A$ if and only if
there is a sheaf of left ideals $\I \subset A_T := A \otimes_{\C} \OO_T$ for some connected scheme $T$ over $\C$
such that $A_T/\I$ is flat over $\OO_T$,
and $\I$ has fibers $\I_t = I$ and $\I_{t'} = I'$ for some points $t, t' \in T( \C)$.

We consider three different cases, depending on the ramification of $A$.

\subsection{No ramification: $e = e' = 1$}
In this case, $A = M_f( R)$ is a full matrix algebra over $R = \C[[u,v]]$. We assume $f > 1$.

\begin{lem} \label{lem:unram_I'}
  Every proper left ideal $I \subset A$ of finite colength can be deformed to
  a proper left ideal $I' \subset A$ of finite colength such that $I'A^* \not\subseteq A^*I'$.
\end{lem}
\begin{proof}
  The left ideal $I \subset A$ is Morita equivalent to an $R$-submodule $M \subset R^f$ of some colength $l < \infty$.
  Choose an ideal $J \subset R$ of colength $l$. Then the $R$-submodule
  \begin{equation} \label{eq:M'}
    M' := J \oplus R^{f-1} \subset R^f
  \end{equation}
  is Morita equivalent to some left ideal $I' \subset A$. Since the punctual Quot-scheme
  \begin{equation*}
    \Quot_R(R^f, l)
  \end{equation*}
  is irreducible according to \cite[Proposition 6]{lehn2}, $M \subset R^f$ can be deformed to $M' \subset R^f$.
  Therefore $I \subset A$ can be deformed to $I' \subset A$. It remains to prove $I'A^* \not\subseteq A^*I'$.

  Assume for contradiction that $I'A^* \subseteq A^*I'$. Then $I'A \subseteq AI'$,
  because $A^* \cong A$ as $A$-bimodules by means of the trace form $A \otimes_R A \to R$.
  Hence $I'$ is a two-sided ideal. Consequently, $I' = M_f( J')$ for some ideal $J' \subset R$. Therefore,
  \begin{equation*}
    M' = (J')^f \subset R^f.
  \end{equation*}
  Since $f > 1$, this contradicts \eqref{eq:M'}. Hence indeed $I'A^* \not\subseteq A^*I'$.
\end{proof}

\subsection{Smooth ramification: $e > e' = 1$}
In this case, our algebra $A = A_{e, 1, f}$ over $R = \C[[u,v]]$ is ramified over $u = 0$,
with ramification index $e$. Explicitly, we have $A = M_f( B)$ for
\begin{equation} \label{eq:B}
  B=\begin{pmatrix*}
    R & \cdots & \cdots & R  \\
    uR & R & \ddots & \vdots\\
    \vdots & \ddots & \ddots & \vdots\\
    uR & \cdots & uR & R
  \end{pmatrix*} \subset M_e(R).
\end{equation}
The aim of this subsection is to prove an analogue of Lemma \ref{lem:unram_I'} in this situation.

We have $A^* = M_f( B^*)$ for the $B$-bimodule $B^* := \Hom_R( B, R)$. The trace map
\begin{equation*}
  \tr: B_K := B \otimes_K R = M_e( K) \to K
\end{equation*}
allows us to identify $B^*$ with the set of all $b \in B_K$ for which $\tr( bB) \subseteq R$; explicitly,
\begin{equation*}
  B^* = \begin{pmatrix*}
    R & u^{-1}R & \cdots & u^{-1}R\\
    R & R & \ddots & \vdots\\
    \vdots & \vdots & \ddots & u^{-1}R\\
    R & R & \cdots & R
  \end{pmatrix*} \subset B_K = M_e( K).
\end{equation*}
In particular, $B^* = b^* B = B b^*$ and $A^* = b^* A = A b^*$ for the matrix
\begin{equation} \label{eq:b^*}
  b^* := \begin{pmatrix*}
    0 & u^{-1} & 0 & \cdots & 0 \\
    \vdots & 0 & u^{-1} & \ddots & \vdots\\
    \vdots & \ddots & \ddots & \ddots & 0\\
    0 & \ddots & \ddots & 0 & u^{-1}\\
    1 & 0 & \cdots & \cdots & 0
  \end{pmatrix*} \in B^*,
\end{equation}
where elements of $A = M_f( B)$ are multiplied componentwise by $b^* \in B^*$. 

We see from \eqref{eq:B} that $B$ has exactly $e$ two-sided maximal ideals $\m_i$,
given by replacing $R$ by its maximal ideal $\m$ in the diagonal entry $(i,i)$ respectively.
So there are also exactly $e$ non-isomorphic simple $B$-modules $S_i := B/\m_i$. We have
\begin{equation*}
  B^* \otimes_B S_1 \cong S_e \quad\text{and}\quad B^* \otimes_B S_i \cong S_{i-1} \quad\text{for } i \geq 2,
\end{equation*}
because $b^* \m_1 = \m_e b^*$ and $b^* \m_i = \m_{i-1} b^*$ for $i \geq 2$, as is easily checked.
Using Morita equivalence, we see that there are $e$ simple left $A$-modules, all of $R$-length $f$.

\begin{cor}\label{hilbsmo}
$Hilb_A(l)$ is nonempty if and only if $f$ divides $l$.
\end{cor}

\begin{lem} \label{lem:two-sided}
  Let $I \subset A$ be a left ideal such that $IA^* \subseteq A^*I$. Then $I$ is a two-sided ideal.
  In particular, $I = M_f( J)$ for some two-sided ideal $J \subset B$ such that $JB^* \subseteq B^*J$.
\end{lem}
\begin{proof}
  Let $b^* \in B^*$ still be the matrix given by \eqref{eq:b^*}. The finitely generated $R$-modules
  $A/I$ and $A^*/A^* I = b^* A/b^* I$ are isomorphic, as $b^*$ is invertible in $M_e( K)$. The $R$-linear map
  \begin{equation*}
    \phi: A/I \to A^*/A^* I, \qquad a + I \mapsto ab^* + A^* I
  \end{equation*}
  is well-defined since $I b^* \subseteq A^* I$ by assumption, and surjective since $Ab^* = A^*$.
  Therefore, $\phi$ is also injective, according to \cite[Theorem 2.4]{mats}. Since $\phi$ is by definition
  $A$-linear from the left, and $I A^* \subseteq A^* I$ by assumption, we conclude that $I A \subseteq I$.
\end{proof}

\begin{lem} \label{lem:J}
  Let $J \subset B$ be a left ideal such that $JB^* \subseteq B^*J$. Then
  \begin{equation} \label{eq:J}
    J=\begin{pmatrix*}
      J_e & J_{e-1} & \cdots & J_1 \\
      u J_1 & J_e & \ddots & \vdots\\
      \vdots & \ddots & \ddots & J_{e-1}\\
      u J_{e-1} & \cdots & u J_1 & J_e
    \end{pmatrix*} \subset M_e( R)
  \end{equation}
  for some chain of ideals $R \supseteq J_1 \supseteq J_2 \supseteq \dots \supseteq J_e$ with $J_e \supseteq u J_1$.
\end{lem}
\begin{proof}
  Since Lemma \ref{lem:two-sided} applies to $J \subset B$, it shows that $J$ is a two-sided ideal in $B$.
  We denote the standard basis elements of the free $R$-module $B$ by
  \begin{equation} \label{eq:b_ij}
    b_{i, j} \in M_e( R), \qquad 1 \leq i, j \leq e.
  \end{equation}
  In other words, the matrix $b_{i, j}$ has a single nonzero entry in row $i$ and column $j$, which is $1$
  for $i \leq j$ and $u$ for $i > j$. Since $J$ is two-sided, we have $b_{i,i} J b_{j,j} \subseteq J$, and therefore
  \begin{equation*}
    b_{i, i} J b_{j, j} = J_{i, j} b_{i, j}
  \end{equation*}
  for some ideals $J_{i, j} \subseteq R$. As $b_{1, 1} + b_{2, 2} + \cdots + b_{e, e} = 1$ in $B$, we conclude that
  \begin{equation*}
    J=\begin{pmatrix*}
      J_{1, 1} & J_{1, 2} & \cdots & J_{1, e} \\
      u J_{2, 1} & J_{2, 2} & \ddots & \vdots\\
      \vdots & \ddots & \ddots & J_{e-1, e}\\
      u J_{e, 1} & \cdots & u J_{e, e-1} & J_{e, e}
    \end{pmatrix*} \subset M_e( R).
  \end{equation*}
  Using this description, the other assumption $Jb^* \subseteq b^*J$ directly implies
  \begin{equation*}
    J_{i, e} \subseteq J_{i+1, 1} \subseteq J_{i+2, 2} \subseteq \dots
      \subseteq J_{e, e-i} \subseteq J_{1, e-i+1} \subseteq J_{2, e-i+2} \subseteq \dots \subseteq J_{i, e}
  \end{equation*}
  for $i = 1, \dots, e$. Hence these inclusions are all equalities, and \eqref{eq:J} holds with $J_i := J_{i, e}$.
  Using \eqref{eq:J}, the assumption $J \supseteq b_{1, 2} J$ directly implies
  $J_1 \supseteq J_2 \supseteq \dots \supseteq J_e \supseteq u J_1$.
\end{proof}

\begin{prop} \label{prop:smooth-ram}
  Every proper left ideal $I \subset A$ of finite colength can be deformed to
  a proper left ideal $I' \subset A$ of finite colength such that $I'A^* \not\subseteq A^*I'$.
\end{prop}
\begin{proof}
  We may assume $IA^* \subseteq A^*I$, since otherwise there is nothing to prove.
  Using Lemma \ref{lem:two-sided} and Lemma \ref{lem:J},
  we get $I = M_f( J)$ with $J \subset B$ given by \eqref{eq:J} for some ideals
  \begin{equation*}
    R \supseteq J_1 \supseteq J_2 \supseteq \dots \supseteq J_e
  \end{equation*}
  of finite colength, not all equal to $R$, such that $J_e \supseteq u J_1$. It suffices to deform $J$ to a
  left ideal $J' \subset B$ such that $J'B^* \not\subseteq B^*J'$. Changing $J$ only in the first row, we will take
  \begin{equation} \label{eq:J'}
    J'=\begin{pmatrix*}
      J_e' & J_{e-1}' & \cdots & J_1' \\
      u J_1 & J_e & \cdots & J_2\\
      \vdots & \ddots & \ddots & \vdots\\
      u J_{e-1} & \cdots & u J_1 & J_e
    \end{pmatrix*} \subset M_e( R)
  \end{equation}
  for some ideals $J_1', \dots, J_e' \subseteq R$, chosen as follows.

  Suppose that $J_1 = \dots = J_e$. Since $\m J_e \neq J_e$ by Nakayama's lemma,
  the vector space $J_e/\m J_e$ over $R/\m = \C$ has a one-dimensional quotient.
  Hence we can find an ideal
  \begin{equation*}
    J_e' \subseteq J_e \quad\text{with}\quad J_e/J_e' \cong \C \quad\text{as $R$-modules.}
  \end{equation*}
  Since $J_1 \neq R$ by assumption, the $R$-module $R/J_1$ of finite length has a simple submodule, which is
  necessarily isomorphic to $R/\m = \C$. Hence we can find an ideal
  \begin{equation*}
    J_1' \supseteq J_1 \quad\text{with}\quad J_1'/J_1 \cong \C \quad\text{as $R$-modules.}
  \end{equation*}
  Finally, we take $J_i' = J_i$ for $i \neq e, 1$ in this case.

  Now suppose that $J_1 = \dots = J_e$ is not true. Choose an index $m$ with $J_m \neq J_{m+1}$.
  Then the $R$-module $J_m/J_{m+1}$ of finite length has a simple submodule and a simple quotient,
  which are both necessarily isomorphic to $R/\m = \C$. Hence we can find two ideals
  \begin{equation*}
    J_m \supseteq J_m', J_{m+1}' \supseteq J_{m+1} \quad\text{with}\quad
      J_m/J_m' \cong \C \cong J_{m+1}'/J_{m+1} \quad\text{as $R$-modules.}
  \end{equation*}
  Finally, we take $J_i' = J_i$ for $i \neq m, m+1$ in this case.

  To show that the $R$-submodule $J' \subseteq B$ defined by \eqref{eq:J'} is a left ideal, we check that
  the basis elements $b_{i, j} \in B$ in \eqref{eq:b_ij} satisfy $b_{i, j} J' \subseteq J'$.
  This clearly holds for $i = j = 1$, and also for $i, j \geq 2$ because $J$ is a left ideal.
  In each of the two cases considered above, the ideals $J_1', \dots, J_e' \subseteq R$ satisfy by construction
  \begin{equation*}
    J_i \subseteq J_{i-1}' \quad\text{for}\quad i \geq 2, \quad\text{and}\quad u J_1 \subseteq J_e'.
  \end{equation*}
  This directly implies $b_{1, 2} J' \subseteq J'$. Similarly, $J_1', \dots, J_e'$ also satisfy by construction
  \begin{equation*}
    J_i' \subseteq J_{i-1} \quad\text{for}\quad i \geq 2, \quad\text{and}\quad u J_1' \subseteq J_e.
  \end{equation*}
  This directly implies $b_{e, 1} J' \subseteq J'$.
  Using $b_{1, i} = b_{1, 2} b_{2, i}$ and $b_{i, 1} = b_{i, e} b_{e, 1}$ for $i \geq 2$, we conclude that
  that $J' \subseteq B$ is indeed a left ideal.

  Since $J'$ is by construction not of the form \eqref{eq:J},
  Lemma \ref{lem:J} shows that $J'B^* \not\subseteq B^*J'$.
  It remains to prove that $J$ can be deformed to $J'$.

  The left $B$-modules $J/(J \cap J')$ and $J'/(J \cap J')$ are, by construction of $J'$, both isomorphic to
  the simple module $S_1 = B/\m_1$. Consequently, the sum $J + J' \subseteq B$ satisfies
  \begin{equation*}
    \dfrac{J + J'}{J \cap J'} \cong \dfrac{J}{J \cap J'} \oplus \dfrac{J'}{J \cap J'}
      \cong S_1 \oplus S_1 \cong \C^2
  \end{equation*}
  where all these $B$-modules are $\C$-vector spaces because $B$ acts on them via
  $B \twoheadrightarrow B/\m_1 \cong \C$.
  We consider the $\PP^1$ of lines in this $\C^2$. The universal quotient
  \begin{equation*}
    \C^2 \otimes_{\C} \OO_{\PP^1} \twoheadrightarrow \OO_{\PP^1}( 1)
  \end{equation*}
  over this $\PP^1$ gives rise to a family of $B$-module quotients 
  \begin{equation*}
    (J+J') \otimes_{\C} \OO_{\PP^1} \twoheadrightarrow \OO_{\PP^1}( 1).
  \end{equation*}
  Its kernel $\J \subset B \otimes_{\C} \OO_{\PP^1}$ restricts
  to $J$ over $[1:0] \in \PP^1$, and to $J'$ over $[0:1] \in \PP^1$.
  Therefore, $\J$ is the required deformation of $J$ to $J'$.
\end{proof}

\subsection{Singular ramification with equal ramification indices: $e = e' > 1$}
In this case, our algebra $A = A_{e, e, f}$ over $R = \C[[u,v]]$ is ramified over $u = 0$ and over $v = 0$,
with common ramification index $e$. Explicitly, we have $A = M_f( S)$ for
\begin{equation*}
  S = R \langle x,y \rangle \text{ with the relations } x^e=u, y^e=v \text{ and } yx=\zeta xy
\end{equation*}
where $\zeta$ is a primitive $e$-th root of unity. The ring $S$ is local in the sense that
it has a unique two-sided maximal ideal $\n \subset S$, which is generated by $x$ and $y$.

In this situation, the analogue of Lemma \ref{lem:unram_I'} is no longer true; a counterexample is given by
$f = 1$ and $I = \n$. However, the following fact will suffice for our purposes.

\begin{lem} \label{lem:sing-ram}
  $\Hilb_A(l)$ is connected if $f$ divides $l$, and it is empty otherwise.
\end{lem}
\begin{proof}
  The unique simple $S$-module $S/\n \cong \C$ has $R$-length one.
  Therefore, $S/\n$ is Morita equivalent to a unique simple left $A$-module, whose $R$-length is $f$. 

  Now one can just copy the corresponding part in the proof of \cite[Theorem 3.6. iii)]{hst}
  and replace the Quot- and the Flag-scheme by the punctual versions.
  The main point is that induction also works in this case, because $A$ has just one simple left module. 
\end{proof}

\section{Moduli spaces of rank one sheaves}\label{sec3}
Let $\A$ be a terminal order on a smooth projective surface $X$ over $\C$.
\begin{defi}[\normalfont{\cite[Definition 4]{chan5}}] \label{def:omega_A}
  The \emph{canonical bimodule} of $\A$ is
  \begin{equation*}
    \omega_{\A} := \HHom_{\OO_X}(\A,\omega_X).
  \end{equation*}
\end{defi}

\begin{lem}[\normalfont{\cite[Theorem 1.58]{reede}}]\label{serredu}
  Let $E$ and $F$ be two $\OO_X$-coherent left $\A$-modules. Then there is the following form of Serre duality:
  \begin{equation*}
    \Ext^i_\A(E,F) \cong \Ext^{2-i}_\A(F,\omega_\A\otimes_\A E)^{\vee}
  \end{equation*}
  for $i\in\{0,1,2\}$. Here $(\_)^{\vee}$ denotes the $\C$-dual.
\end{lem}

\begin{lem}[\normalfont{\cite[Lemma 1.62]{reede}}]\label{forgetin}
  Let $E$ and $T$ be $\OO_X$-coherent left $\A$-modules such that $E$ is locally projective
  and $T$ is an Artinian module of finite length. Then the map
  \begin{equation*}
    \Ext_\A^2(T,E) \to \Ext_{\OO_X}^2(T,E)
  \end{equation*}
  induced by the forgetful functor $\A \lmod \to \OO_X \lmod$ is injective.
\end{lem}

\begin{defi}
  A left $\A$-module $E$ is called a \emph{torsion-free $\A$-module of rank one} if
  \begin{itemize}
   \item $E$ is coherent and torsion-free as an $\OO_X$-module, and
   \item the stalk $E_{\eta}$ at the generic point $\eta \in X$ has dimension $1$ over the division ring $\A_{\eta}$.
  \end{itemize}
\end{defi}

\begin{lem}[\normalfont{\cite[Proposition 4.2.]{chan6}}]\label{projfree}
  Let $E$ be a torsion-free $\A$-module of rank one which is a locally free $\OO_X$-module,
  then for every closed point $p\in X$ there is an isomorphism of completions 
  \begin{equation*}
    \Ehat_p\cong\Ahat_p.
  \end{equation*}
  Thus $E$ is locally free over $\A$ if and only if $E$ is locally free over $\OO_X$.
\end{lem}

\begin{lem}[\normalfont{\cite[Theorem 1.84]{reede}}]\label{ceinsd}
  If $E$ is a torsion-free $\A$-module of rank one, then
  \begin{equation*}
    c_1(\A^* \otimes_{\A} E)=c_1(E)-2c_1(\A)
  \end{equation*}
  where $\A^* := \HHom_{\OO_X}( \A, \OO_X)$ denotes the dual sheaf of $\A$.
\end{lem}

\begin{defi}
  A \emph{family of torsion-free $\A$-modules of rank one} over a $\C$-scheme $T$ is a left module $\E$ under the
  pullback $\A_T$ of $\A$ to $X \times T$ with the following properties:
  \begin{itemize}
   \item $\E$ is coherent over $\OO_{X\times T}$ and flat over $T$;
   \item for every $t \in T$, the fiber $\E_t$ is a torsion-free $\A_{\C(t)}$-module of rank one.
  \end{itemize}
  Here $\C(t)$ is the residue field of $T$ at $t$, and the fiber is the pullback of $\E$ to $X \times \Spec \C(t)$. 
\end{defi}
Now one can define the moduli functor
\begin{equation*}
  \M_{\A/X:P}: \Schemes_{\C} \to \Sets
\end{equation*}
which sends a $\C$-scheme $T$ to the set of isomorphism classes of families $\E$
of torsion-free $\A$-modules of rank one over $T$ with Hilbert polynomial $P$. 
\begin{thm}[\normalfont{\cite[Theorem 2.4]{hst}}]
  There is a coarse moduli scheme $M_{\A/X;P}$ for the functor $\M_{\A/X;P}$.
  The scheme $M_{\A/X;P}$ is of finite type and projective over $\C$.
\end{thm}

Instead of fixing the Hilbert polynomial, one can also fix the Chern classes of these modules.
We will work with the moduli space $M_{\A/X;c_1,c_2}$ of torsion-free $\A$-modules of rank one over $X$
with Chern classes $c_1 \in \NS(X)$ and $c_2\in \Z$.

\begin{lem} \label{extvan}
  Let $\A$ be a terminal del Pezzo order on $\PP^2$ over $\C$.
  If $E$ and $F$ are torsion-free $\A$-modules of rank one with $c_1(E)=c_1(F)$, then $\Ext^2_\A(E,F)=0$.
\end{lem}
\begin{proof}
  Assume for contradiction that $\Ext^2_\A(E,F) \neq 0$. Then Serre duality for $\A$-modules
  states that there is a nonzero map
  \begin{equation*}
    \phi: F \to \omega_\A \otimes_\A E.
  \end{equation*}
  Since $E$ and $F$ are generically simple and torsion-free, $\phi$ is generically bijective and therefore
  injective, and its cokernel is a torsion sheaf. This means that the divisor class
  \begin{equation} \label{eq:diff}
    c_1( \omega_\A \otimes_{\A} E) - c_1( F)
  \end{equation}
  is effective. On the other hand, Definition \ref{def:omega_A}, Lemma \ref{ceinsd} and Lemma \ref{lem:K_A} imply that
  \begin{align*}
    c_1( \omega_\A \otimes_{\A} E) & = c_1( \A^{*} \otimes_{\A} E) + \rk(\A) c_1( \omega_{\PP^2})\\
      & = c_1( E) - 2c_1( \A) + \rk(\A) K_{\PP^2}\\
      & = c_1( E) + \rk(\A) K_{\A}.
  \end{align*}
  Hence the class in \eqref{eq:diff} equals $\rk(\A) K_{\A}$.
  But $\A$ is a del Pezzo order, so $-K_\A$ is ample.
  Since $\Pic( \PP^2) = \Z \cdot [\OO( 1)]$, we conclude that $K_\A$ and the class in \eqref{eq:diff}
  are negative multiples of $[\OO( 1)]$, and therefore not effective.
  This contradiction proves $\Ext^2_\A(E,F)=0$.
\end{proof}

\begin{thm}\label{modsmo}
  If $\A$ is a terminal del Pezzo order on $\PP^2$ over $\C$, then the moduli space $M_{\A/\PP^2;c_1,c_2}$
  of torsion-free $\A$-modules of rank one with Chern classes $c_1$ and $c_2$ is smooth.
\end{thm}
\begin{proof}
  Let $E$ be a torsion-free $\A$-module of rank one with Chern classes $c_1$ and $c_2$.
  Then $\Ext^2_\A(E,E)=0$ according to Lemma \ref{extvan}.
  In particular, all obstruction classes in $\Ext^2_\A(E,E)$ vanish.
  This implies that $M_{\A/\PP^2;c_1,c_2}$ is smooth at the point $[E]$.
\end{proof}

\section{Deformations of torsion-free rank one sheaves}\label{sec4}
Let $\A$ be a terminal del Pezzo order of rank $n^2 > 1$ on the projective plane $\PP^2$ over $\C$.
Let $D \subset \PP^2$ denote the ramification divisor. Proposition \ref{prop:delpez} states that
$\A$ has the same ramification index $e$ at every component of $D$. We put $f := n/e$.
\begin{prop} \label{prop:pi'}
  Let $E$ be a locally free left $\A$-module of rank one. Let
  \begin{equation} \label{eq:pi}
    \pi: E \twoheadrightarrow T
  \end{equation}
  be a nonzero quotient of finite length. Then $\pi$ can be deformed to a nonzero quotient
  \begin{equation} \label{eq:pi'}
    \pi': E \twoheadrightarrow T'
  \end{equation}
  of finite length such that the following induced map is not injective:
  \begin{equation} \label{eq:pi_*'}
    \pi_*': \Ext^2_\A( T', E) \to \Ext^2_\A( T', T')
  \end{equation}
\end{prop}
\begin{proof}
  Choose $p \in \PP^2$ in the support of $T$. As $T$ has finite length, its support is finite, and
  \begin{equation*}
    T = T_p \oplus T_{\neq p}
  \end{equation*}
  where $T_p$ is supported at $p$, and $T_{\neq p}$ is supported outside $p$.
  We distinguish three cases, depending on the ramification of $\A$ at $p$.

  The first case is that $p$ is a smooth point of the ramification divisor $D$.
  Let $A := \Ahat_p$ denote the completion of $\A$ at $p$, that is we have $A\cong A_{e,1,f}$.
  Choosing an isomorphism of completions given by Lemma \ref{projfree}
  \begin{equation} \label{eq:Ehat_p}
    \Ehat_p \cong A,
  \end{equation}
  we can identify the quotient $T_p$ of $\Ehat_p$ with $A/I$ for some left ideal $I \subset A$ of finite colength.
  Proposition \ref{prop:smooth-ram} allows us to deform $I$ to a left ideal $I' \subset A$ of finite colength
  such that
  \begin{equation*}
    I' A^* \not\subseteq A^* I'.   
  \end{equation*}
  Therefore, $T_p$ can be deformed to $T_p' := A/I'$ as a quotient of $A$, and 
  the given quotient $\pi$ in \eqref{eq:pi} can be deformed to the quotient
  \begin{equation*}
    \pi': E \twoheadrightarrow T' := T_p' \oplus T_{\neq p}.
  \end{equation*}
  To prove that $\pi_*'$ in \eqref {eq:pi_*'} is not injective,
  we choose an element $a^* \in A^*$ with $I' a^*  \not\subseteq A^* I'$.
  Then the left $A$-module homomorphism
  \begin{equation*}
    \phi: A \to A^*/A^* I' = A^* \otimes_A T_p', \qquad a \mapsto aa^* + A^* I',
  \end{equation*}
  does not vanish on $I'$, and hence does not factor through $A/I' = T_p'$. Therefore, the map
  \begin{equation*}
    \Hom_A( T_p', A^* \otimes_A T_p') \to \Hom_A( A, A^* \otimes_A T_p')
  \end{equation*}
  induced by the projection $A \twoheadrightarrow T_p'$ is not surjective, as its image does not contain $\phi$.
  Using the identification \eqref{eq:Ehat_p} and the decomposition $T' = T_p' \oplus T_{\neq p}$, we conclude that
  \begin{equation*}
    (\pi')^*: \Hom_{\A}( T', \omega_{\A} \otimes_{\A} T') \to \Hom_{\A}( E, \omega_{\A} \otimes_{\A} T')
  \end{equation*}
  is not surjective. Hence the map $\pi_*'$ in \eqref{eq:pi_*'} is not injective, by Serre duality for $\A$-modules.

  The second case is that $\A$ is unramified at $p$. This case is simpler than the first case.
  However, the same argument works, using Lemma \ref{lem:unram_I'} instead of Proposition \ref{prop:smooth-ram}.

  The third case is that $p$ lies in the singular locus $D^{\sing}$ of the ramification divisor $D$.
  Let $l$ be the $\OO_{\PP^2}$-length of $T_p$. Then $\pi_p: E \twoheadrightarrow T_p$ defines a point in the scheme
  \begin{equation*}
    \Quot_{\A}(E,l)
  \end{equation*}
  that classifies left $\A$-module quotients of $E$ with $\OO_{\PP^2}$-length $l$.
  This is a closed subscheme of $\Quot_{\OO_{\PP^2}}(E, l)$, and hence projective over $\C$.
  It comes with a Hilbert-Chow morphism
  \begin{equation} \label{eq:supp}
    \supp: \Quot_{\A}(E,l) \to \Sym^l( \PP^2),
  \end{equation}
  whose fiber over $l \cdot q$ for $q \in \PP^2$ is the punctual Hilbert scheme for the completion $\Ahat_q$:
  \begin{equation} \label{eq:supp(lq)}
    \supp^{-1}( l \cdot q) = \Hilb_{\Ahat_q}( l).
  \end{equation}
  For $q = p$, this fiber contains the point $T_p$, and is therefore non-empty.
  Using Lemma \ref{lem:sing-ram}, we conclude that $f$ divides $l$.
  Hence \eqref{eq:supp(lq)} is non-empty for each ramified point $q \in D$ by Corollary \ref{hilbsmo}.
  In other words, the image of the morphism $\supp$ in \eqref{eq:supp} contains the diagonally embedded
  \begin{equation*}
    D \subset \PP^2 \hookrightarrow \Sym^l( \PP^2).
  \end{equation*}
  Let $\Delta \subset D$ be the finite set of all points $q \neq p$ in $D^{\sing}$ or in the support of $T$.
  Choose an irreducible component $C \subseteq D \setminus \Delta$ with $p \in C$. Let $Q_i$ be the connected
  components of
  \begin{equation*}
    \supp^{-1}( C) \subseteq \Quot_{\A}(E,l).
  \end{equation*}
  Since the morphism $\supp$ in \eqref{eq:supp} is projective, the image $\supp( Q_i)$ is closed in $C$.
  But the union of these images is all of $C$, which is irreducible. Hence
  \begin{equation*}
    \supp( Q_i) = C
  \end{equation*}
  for some such connected component $Q_i$. Since $\supp^{-1}( l \cdot p)$ is connected by Lemma \ref{lem:sing-ram},
  and intersects $Q_i$ by construction, it is contained in $Q_i$. In particular, the point given by
  \begin{equation} \label{eq:T_p}
    \pi_p: E \twoheadrightarrow T_p
  \end{equation}
  lies in $Q_i$. Now choose a point $q \neq p$ in $C$, and a quotient
  \begin{equation} \label{eq:T_q'}
    \pi_q': E \twoheadrightarrow T_q'
  \end{equation}
  corresponding to a point in $Q_i$ over $q$. The restriction of the universal quotient to
  \begin{equation*}
    Q_i \subset \Quot_{\A}(E,l)
  \end{equation*}
  provides a deformation of the quotient \eqref{eq:T_p} to the quotient \eqref{eq:T_q'}.
  Since $\supp( Q_i) = C \subset \PP^2$ does not intersect the support of $T_{\neq p}$,
  we can take the direct sum with the component
  \begin{equation*}
    \pi_{\neq p}: E \twoheadrightarrow T_{\neq p}
  \end{equation*}
  of $\pi$ to obtain a deformation of the given quotient \eqref{eq:pi} to the quotient
  \begin{equation*}
    \pi_q' \oplus \pi_{\neq p}: E \twoheadrightarrow T_q' \oplus T_{\neq p}.
  \end{equation*}
  As the support of this quotient contains the point $q \in D \setminus D^{\sing}$, we can 
  apply the first case treated above to deform it further to a quotient \eqref{eq:pi'} with the required property.
\end{proof}

\begin{thm} \label{thm:main}
  Let $\A \neq \OO_{\PP^2}$ be a terminal del Pezzo order on $\PP^2$ over $\C$.
  Then every torsion-free $\A$-module $E$ of rank one can be deformed to a locally free $\A$-module $E'$.
\end{thm}
\begin{proof}
  We adapt the proof of \cite[Theorem 3.6.(iii)]{hst} and start with the exact sequence
  \begin{equation}\label{seq1}\begin{CD}
    0 @>>> E @>\iota>> E^{**} @>\pi>> T @>>> 0
  \end{CD}\end{equation}
  induced by $E$. The functor $\Hom_{\A}(T, \_)$ turns \eqref{seq1} into the long exact sequence
  \begin{equation*}\begin{CD}
    \ldots @>>> \Ext_{\A}^2(T, E) @>\iota_{*}>> \Ext_{\A}^2(T, E^{**}) @>\pi_{*}>> \Ext_{\A}^2(T, T) @>>> 0.
  \end{CD}\end{equation*} 
  Applying Proposition \ref{prop:pi'} to the quotient $\pi: E^{**} \twoheadrightarrow T$,
  and replacing $E$ by the kernel of the resulting deformed quotient $\pi': E^{**} \twoheadrightarrow T'$,
  we may assume that $\pi_*$ is not injective. Then $\iota_* \neq 0$.
  The functor $\Hom_{\A}(\_, E)$ turns \eqref{seq1} into the long exact sequence
  \begin{equation*}\begin{CD}
    \ldots @>>> \Ext_\A^1(E, E) @>\partial>> \Ext_\A^2(T, E) @>>> \Ext_{\A}^2(E^{**}, E) @>>> \ldots
  \end{CD}\end{equation*}
  whose connecting homomorphism $\partial$ is surjective by Lemma \ref{extvan}. Hence the composition
  \begin{equation*}\begin{CD}
    \Ext_\A^1(E, E) @>\partial>> \Ext_\A^2(T, E) @>\iota_{*}>> \Ext_{\A}^2(T, E^{**})
  \end{CD}\end{equation*}
  is nonzero. We choose a class $\gamma \in \Ext_\A^1(E, E)$ whose image in $\Ext_\A^2(T, E^{**})$ is nonzero.
  The infinitesimal deformation of $E$ given by $\gamma$ can be extended to a deformation $\E$ of $E$
  over a smooth connected curve $C$, since $\Ext^2_{\A}( E, E) = 0$ by Lemma \ref{extvan}.

  Let $E'$ be the fiber of $\E$ over a general point of $C$. Lemma \ref{forgetin} states that the forgetful functor
  induces an injective map
  \begin{equation*}
    \Ext_\A^2(T, E^{**}) \hookrightarrow \Ext_{\OO_{\PP^2}}^2(T, E^{**}).
  \end{equation*}
  So the class $\gamma$, seen as an element in $\Ext_{\OO_{\PP^2}}^1(E, E)$,
  has nonzero image in $\Ext_{\OO_{\PP^2}}^2(T, E^{**})$.

  We can thus use a result of Artamkin, which says that the length of $(E')^{**}/E'$
  is strictly smaller than the length of $E^{**}/E$, see \cite[Corollary 1.3]{artam}.
  Using induction over this length, we may assume that $E'$ can already be deformed to a locally free $\A$-module. 
\end{proof}

\begin{cor}
  Every irreducible component of the moduli space $M_{\A/\PP^2;c_1,c_2}$
  contains a point defined by a locally free $\A$-module.
\end{cor}
\begin{proof}
  Every connected component of $M_{\A/\PP^2;c_1,c_2}$ contains such a point by Theorem \ref{thm:main}.
  But these connected components are smooth by Theorem \ref{modsmo}, and hence irreducible.
\end{proof}

\begin{cor}
  The open locus $M_{\A/\PP^2;c_1,c_2}^{\lf}$ of locally free $\A$-modules is dense in $M_{\A/X;c_1,c_2}$.
\end{cor}

%\addcontentsline{toc}{section}{References}
\bibliography{Artikel}
\bibliographystyle{alphaurl}

\end{document}